\documentclass[12pt]{amsart}

\usepackage{mathrsfs}

\newtheorem{theorem}{Theorem}[section]
\newtheorem{lemma}[theorem]{Lemma}
\newtheorem{proposition}[theorem]{Proposition}

 \theoremstyle{definition}

\theoremstyle{remark}
\newtheorem{remark}[theorem]{Remark}

\numberwithin{equation}{section}

\usepackage{amssymb}

\usepackage{cases}

\newcommand{\bM}{\bar{M}}

\newcommand{\ul}{\underline}

\begin{document}
\setlength{\baselineskip}{1.2\baselineskip}

\title[Obstacle problem on Riemannian manifolds]
{{\it A priori} estimates for the obstacle problem of Hessian type equations on Riemannian manifolds}

\author{Tingting Wang}
\author{WeiSong Dong}
\author{Gejun Bao}

\address{Department of Mathematics, Harbin Institute of Technology, Harbin 150001, China}

\email{ttwanghit@gmail.com}
\email{dweeson@gmail.com}
\email{baogj@hit.edu.cn}

\begin{abstract}
We are concerned with {\it a priori} estimates for the obstacle problem of a wide class of fully nonlinear equations on Riemannian manifolds.
We use new techniques introduced by Bo Guan and derive new results for {\it a priori} second order estimates of its singular perturbation problem under fairly general conditions.
By approximation, the existence of a $C^{1,1}$ viscosity solution is proved.



\end{abstract}

\maketitle

\bigskip

\section{Introduction}

\medskip

This is one of a series of papers in which we study the obstacle problem for Hessian type equations on Riemannian manifolds.
Let $(M^n, g)$ be a compact Riemannian manifold of
dimension $n \geq 2$ with smooth boundary $\partial M$, $\bM := M \cup \partial M$, and $\nabla$ denote its Levi-Civita connection.
In this paper we study the obstacle problem
\begin{equation}
\label{11}
\max \left\{u -h, - (f (\lambda (\nabla^2 u + A [u])) - \psi [u]) \right\} = 0 \; \mbox{ in } M
\end{equation}
with the Dirichlet boundary condition
\begin{equation}
\label{111}
u = \varphi \; \mbox{ on } \partial M,
\end{equation}
where $h \in C^3 (\bM)$ is called an obstacle, $\varphi \in C^4 (\partial M)$,
$h > \varphi$ on $\partial M$, $\psi [u] = \psi (x,  u, \nabla u)$ is a positive function of $C^3$ and $A [u] = A (x, u, \nabla u)$ is a smooth $(0,2)$ tensor which may depend on $u$ and $\nabla u$, $f$ is a symmetric function of $\lambda \in \mathbb{R}^n$,
and for a (0, 2) tensor $X$ on $M$,
$\lambda (X)$ denotes the eigenvalues of $X$ with respect to the metric $g$.

Following \cite{CNS}, the function
$f \in C^2 (\Gamma) \cap C^0 (\overline{\Gamma})$
is assumed to be defined in an open, convex, symmetric
cone $\Gamma \subsetneqq \mathbb{R}^{n}$, with vertex at the origin, containing the positive cone:
$\{\lambda\in\mathbb{R}^{n}: \mbox{each component } \lambda_{i}>0\}$,
and to satisfy the fundamental structure conditions
\begin{equation}
\label{3I-20}
f_{i} \equiv \frac{\partial f}{\partial \lambda_{i}} > 0
\mbox{ in }\Gamma,\  1 \leq i \leq n,
\end{equation}
\begin{equation}
\label{3I-30}
\mbox{$f$ is a concave function in $\Gamma$},
\end{equation}
and
\begin{equation}
\label{3I-40}
f > 0 \mbox{ in } \Gamma, ~~ f = 0 \mbox{ on } \partial \Gamma.
\end{equation}

A function $u \in C^2 (M)$ is called {\it admissible} at $x \in M$
if $\lambda (\nabla^2 u + A [u]) (x) \in \overline{\Gamma}$
and we call it {\it admissible} in $M$ if it is admissible at
each $x \in M$.
It is shown in \cite{CNS} that
\eqref{3I-20} implies that \eqref{11} is elliptic for admissible
solutions, and \eqref{3I-30} ensures
that $F$ defined by $F (r) = f (\lambda (r))$
for $r = \{r_{ij}\} \in \mathcal{S}^{n \times n}$ with $\lambda (r) \in \Gamma$
is concave, where $\mathcal{S}^{n \times n}$ is the set of $n \times n$ symmetric matrices.

In this paper, we prove the existence of a viscosity solution of \eqref{11} and \eqref{111} in $C^{1, 1} (\bM)$ (see \cite{CIL,T} for the definition of viscosity solution).
Our motivation to study equation \eqref{11}
comes partly from its geometric applications.
In \cite{GC0} Gerhardt considered hypersurfaces having prescribed mean curvature $H$ that are bounded from below by an obstacle. The case $H=0$ (minimal surfaces) had been studied by for example Kinderlehrer \cite{Kinderlehrer71,Kinderlehrer73} and Giusti \cite{Giusti}.
Xiong and Bao \cite{XB} studied the problem of finding the greatest hypersurface below a given obstacle, whose
Gauss-Kronecker curvature (accordingly, $f = \sigma^{1/n}_n$) is bounded from below by a positive function, and established  $C^{1,1}$ regularity in nonconvex domains in $\mathbb{R}^n$.
Lee \cite{L} considered obstacle problem for Monge-Amp\`{e}re equation of the case when $A \equiv 0$,
$\psi \equiv 1$, $\varphi \equiv 0$, and proved the $C^{1,1}$ regularity of the viscosity solution and $C^{1,\alpha}$ regularity of free
boundary in a strictly convex domain in $\mathbb{R}^n$.
The interest to \eqref{11} is also arising from its connection to optimal transportation problem, see e.g. Savin \cite{Savin04,Savin05}, Caffarelli and McCann \cite{CM}. Moreover, Liu and Zhou \cite{LZ} treated
an obstacle problem for Monge-Amp\`{e}re equation related to the affine maximal surface
equation and Abreu's equation.
Oberman \cite{O,OS} showed that the convex envelope is a viscosity solution of a partial differential equation in the form of a nonlinear obstacle problem.

The obstacle problem for Hessian equations
on Riemannian manifolds has been studied by Jiao and Wang \cite{JW}, where they considered the case when $A \equiv \kappa u g$ under conditions on $f$ which however exclude the case that $f = (\sigma_k / \sigma_l)^{1/(k - l)}$, $1 \leq l < k \leq n$. Bao, Dong and Jiao \cite{BDJ} considered \eqref{11} and \eqref{111} under a condition (see the condition (2.4) in \cite{BDJ}, see also \cite{GOLD}) which was essential for {\it a priori} second order estimates.
Recently, Jiao \cite{J15} studied an obstacle problem for Hessian equations on Riemannian manifolds using the ideas from the theory of 
the \emph{a priori} estimates for fully nonlinear elliptic equations introduced by Guan \cite{GNEW} (see \cite{GJNEW} for a general form).
Compared with these, we study the obstacle problem of the general case \eqref{11} and \eqref{111}, and derive {\it a priori} estimates without such a condition, using the new technique introduced by Guan \cite{GNEW}, see also \cite{GSSNEW,GJNEW}. Moreover, our problem \eqref{11} covers the case that $f = (\sigma_k / \sigma_l)^{1/(k - l)}$, $1 \leq l < k \leq n$.

\medskip

\textbf{Acknowledgments:}  The authors would like to thank Heming Jiao for drawing the authors' attention to the work about the obstacle problem on Riemannian manifolds and many useful suggestions and comments. We also thank him for sending us his preprint \cite{J15}.

\section{Beginning of Proof}

We use ideas from \cite{GNEW}, see also \cite{GSSNEW,GJNEW}. Suppose, in addition to \eqref{3I-20}-\eqref{3I-40}, that there exists
an admissible subsolution $\ul{u} \in C^2 (\bM)$ satisfying
\begin{equation}
\label{3I-11s}
\left\{
\begin{aligned}
f (\lambda (\nabla^2 \ul{u} + A [\ul u]))  \,& \geq \psi [\ul u] \; \mbox{ in } M, \\
                                     \ul u \,&  = \varphi \; \mbox{ on } \partial M,
\end{aligned}
\right.
\end{equation}
and $\ul u \leq h$ in $M$. We remark here that the existence of $\ul u$ in some special cases can be found in \cite{JW}.

To prove the existence of viscosity solutions to \eqref{11} and \eqref{111},
we use a penalization technique and consider the following singular perturbation problem
\begin{equation}
\label{jw-21}
\left\{
\begin{aligned}
f (\lambda ( \nabla^2 u  + A [u] )) \, & = \psi [u] + \beta_\varepsilon (u-h) \;\; \mbox{ in }  M,\\
                                   u \,& = \varphi \;\; \mbox{ on }  \partial M,
\end{aligned}
\right.
\end{equation}
where the penalty function $\beta_\varepsilon$ is defined by
\begin{equation}
\label{eqn-penal}
\beta_\varepsilon (z) =
\left\{
\begin{aligned}
& 0, & z \leq 0,\\
      & z^3 / \varepsilon, & z > 0,
\end{aligned}
\right.
\end{equation}
for $\varepsilon \in (0, 1)$. Obviously, see \cite{XB}, $\beta_\varepsilon \in C^2 (\mathbb{R})$ satisfies
\begin{equation}
\label{eqn-penalp}
\begin{aligned}
& \beta_\varepsilon, \beta'_\varepsilon, \beta''_\varepsilon \geq 0;\\
& \beta_\varepsilon (z) \rightarrow \infty \; \mbox{ as } \varepsilon \rightarrow 0^+,
     \mbox{ whenever } z > 0;\\
& \beta_\varepsilon (z) = 0, \; \mbox{ whenever } z \leq 0.
\end{aligned}
\end{equation}
Observe that $\ul u$ is also a subsolution to \eqref{jw-21}.

Let
\[
\mathscr{U}
= \left\{ u_{\varepsilon} | ~u_{\varepsilon} \in C^4(\bar M) \; \mbox{is an admissible solution of} \; \eqref{jw-21} \mbox{ with } u_{\varepsilon} \geq \ul u \mbox{ on } \bM \right\}.
\]
We aim to derive the uniform bound
\begin{equation}
\label{maines}
|u_\varepsilon|_{C^2 (\bM)} \leq C
\end{equation}
for $u_\varepsilon \in \mathscr{U}$, where $C$ is independent of $\varepsilon$. Once \eqref{maines} is obtained, we conclude that there exists a function $C^{1, 1} (\bM)$ satisfying \eqref{11} and \eqref{111}, see \cite{BDJ,XB}.

\begin{remark}
For simplicity, we may drop the subscript $\varepsilon$ in the following when
there is no possible confusion.
\end{remark}

In the proof of the second order estimates, we adapt new methods introduced by Guan \cite{GNEW}. We use notations in \cite{GNEW}. Write $\mu(x)=\lambda(\nabla^2 \ul u (x) + A[\ul u](x))$ and note that $\{\mu (x): x \in \bar M\}$
is a compact subset of $\Gamma$. For all $\lambda \in \Gamma$, let $\nu_{\lambda} = D f(\lambda) / |D f (\lambda)|$ denote the unit normal vector to the level hypersurface of $f$ through $\lambda$. There exists a uniform constant $\zeta_0 \in (0, \frac{1}{2 \sqrt{n}})$
such that
\begin{equation}
\label{zeta0}
\nu_{\mu (x)} - 2 \zeta_0 \mathbf{1} \in \Gamma_n, \; \forall x \in \bar M
\end{equation}
where $\mathbf{1} = (1, \cdots, 1) \in \mathbb{R}^n$.

We need the following lemma which is crucial in deriving {\it a priori} $C^2$ estimates.
\begin{lemma}[\cite{GNEW,GJNEW}]
\label{gnew-lem}
Let $K$ be a compact subset of $\Gamma$ and $\zeta>0$. There is a constant $\theta>0$ such that for any $\mu\in K$ and $\lambda\in\Gamma$, when
$|\nu_{\mu} - \nu_{\lambda}| \geq \zeta$,
\begin{equation}
\label{gnew}
\sum f_i (\lambda) (\mu_i - \lambda_i  ) \geq f (\mu) - f (\lambda) + \theta(1 + \sum f_i (\lambda)).
\end{equation}
\end{lemma}

We use the notation
\[
A^{\xi \eta} (x, \cdot, \cdot) := A (x, \cdot, \cdot) (\xi, \eta), \;\;
\xi,  \eta \in T^*_x M.
\]
and $U:=\nabla^2 u +A[u]$, $F(U) = f (\lambda (U))$. Under a local frame $e_{1}, \ldots, e_{n}$, $U_{ij}:=U(e_i,e_j)=\nabla_{ij} u +A^{ij}[u]$ and
\[
F^{ij} = \frac{\partial F}{\partial U_{ij}} (U), \;\;
F^{ij, kl} = \frac{\partial^2 F}{\partial U_{ij} \partial U_{kl}} (U).
\]
Let $\mathcal{L}$ be the linear operator locally defined by
\[
\mathcal{L} v := F^{ij} \nabla_{ij} v
    + (F^{ij} A^{ij}_{p_k} - \psi_{p_k}) \nabla_k v, \;\; v \in C^2 (M).
\]

In the process of deriving {\it a priori} second order estimates, see Section 3 below, we apply Lemma \ref{gnew-lem} with $\zeta=\zeta_0$ in \eqref{zeta0} (we will explain this in Remark \ref{rmk}),
and an immediate result shows that:
\begin{proposition}
\label{prop-L}
Assume that
\begin{equation}
\label{A2}
\mbox{$-\psi (x,z,p)$ and $A^{\xi \xi} (x,z,p)$ are concave in $p$},
\end{equation}
\begin{equation}
\label{A4}
- \psi_z, \; A^{\xi \xi}_{z}  \geq 0, \;\; \forall \, \xi \in T_x M.
\end{equation}
Then if $|\nu_{\mu} - \nu_{\lambda}| \geq \zeta_0$, we have
\begin{equation}
\label{L-u}
\mathcal{L} (\ul u - u) \geq \theta (1 + \sum F^{ii}) - \beta_{\varepsilon} (u - h).
\end{equation}
\end{proposition}

\begin{proof}
For any $x\in M$, choose a smooth orthonormal local frames $e_1, \ldots, e_n$ about $x$ such that $\{U_{ij} (x)\}$ is diagonal, so is $\{F^{ij} (U) (x)\}$.
If $|\nu_{\mu} - \nu_{\lambda}| \geq \zeta_0$, then by Lemma \ref{gnew-lem}, we have
\[
F^{ii} (U) (\ul U_{ii} - U_{ii})  \geq  \psi [\ul u] - \psi [u] - \beta_\varepsilon (u - h) + \theta (1 + \sum F^{ii}).
\]
It follows from \eqref{A2} and \eqref{A4} that
\begin{equation}
\label{concave}
A^{ii}_{p_k} \nabla_k (\ul u- u)  \geq  A^{ii} [\ul u] - A^{ii}[u] \; \mbox{ and } - \psi_{p_k} \nabla_k (\ul u - u)  \geq  - \psi [\ul u] + \psi [u].
\end{equation}
Thus \eqref{L-u} is obtained.
\end{proof}

\begin{remark}
\label{rmk}
In another case $|\nu_{\mu} - \nu_{\lambda}| < \zeta_0$, we have by \eqref{zeta0} that $\nu_{\lambda} - \zeta_0 \mathbf{1} \in \Gamma_n$, and therefore
\begin{equation}
\label{F'}
F^{ii} \geq \frac{\zeta_0}{\sqrt{n}} \sum F^{kk}, \; \forall 1 \leq i \leq n.
\end{equation}
We also have in this case that, by the concavity of $F$,
\[
F^{ii} (U) (\ul U_{ii} - U_{ii}) \geq F(\ul U) - F(U) \geq  \psi [\ul u] - \psi [u] - \beta_\varepsilon (u - h)
\]
Then combining with \eqref{concave} we obtain
\begin{equation}
\label{L'}
\mathcal{L} (\ul u - u) \geq - \beta_{\varepsilon} (u - h).
\end{equation}
\end{remark}

\begin{remark}
Note that \eqref{L-u} and \eqref{L'} are the highlight of the paper.
\end{remark}

\section{Estimates for second order derivatives}

In this section, we prove {\it a priori} estimates of second order derivatives for an admissible solution $u \in \mathscr{U}$.
We see that $\mathrm{tr} (A [u]) \leq C$ on $\bar{M}$, where $C$ is independent of $\varepsilon$ and $C$ depends on $|u|_{C^1(\bar M)}$.
Let $G$ be the solution to
\[
\left\{
\begin{aligned}
\Delta G + C  \,& = 0 \; \mbox{ in } M, \\
            G \,& = \varphi \; \mbox{ on } \partial M,
\end{aligned}
\right.
\]
Then we have $u \leq G$ in $M$ by the maximum principle
since $\Delta u + C \geq \Delta u + \mathrm{tr} (A [u]) > 0$ in $M$.
Since $h > \varphi$ on $\partial M$, we have $h > G \geq u$ in a neighborhood of $\partial M$
in which $\beta_\varepsilon (u - h) \equiv 0$. Thus, in such a neighborhood of $\partial M$, the Dirichlet problem \eqref{jw-21} reduces to
\begin{equation}
\label{jw-21'}
\left\{
\begin{aligned}
f (\lambda ( \nabla^2 u  + A [u] )) \, & = \psi [u] \;\; \mbox{ in a neighborhood of } \partial M,\\
                                   u \,& = \varphi \;\; \mbox{ on }  \partial M,
\end{aligned}
\right.
\end{equation}
and hence by the arguments of Section 3 in \cite{GJNEW}, we obtain the boundary estimates for second order derivatives
\begin{equation}
\label{eqn-c2b}
|\nabla^2 u| \leq C \;\; \mbox{  on } \partial M
\end{equation}
under assumptions \eqref{3I-20}-\eqref{3I-40}, \eqref{3I-11s}, \eqref{A2}, \eqref{A4}, and
\begin{equation}
\label{f7}
\sum f_i (\lambda) \lambda_i \geq - K_0 (1+ \sum f_i), \;\;  \forall \lambda \in \Gamma,
\end{equation}
for some $K_0 \geq 0$, where the constant $C$ in \eqref{eqn-c2b} is independent of $\varepsilon$ and depends on $|u|_{C^1(\bar M)}$. Note that the condition \eqref{f7} is used to overcome the difficulty caused by the
presence of curvature in the boundary estimates \eqref{eqn-c2b} (see \cite{GOLD,GNEW,GJNEW}).

Therefore, it remains to estimate the interior second order derivatives $|\nabla^2 u|_{C^0 (M)}$ for the global estimates of second derivatives $|\nabla^2 u|_{C^0 (\bM)}$.
The following lemma will be needed which is key in both the second derivative estimates and the gradient estimates.
\begin{lemma}[\cite{BDJ,XB}]
\label{jw-lem1}
There exists a positive constant $c_0$, which is independent of $\varepsilon$ and depends on $|u|_{C^0(\bM)}$, such that
\begin{equation}
\label{jw-31}
0\leq \beta_\varepsilon (u-h) \leq c_0   \; \mbox{ in } M.
\end{equation}
\end{lemma}

Now we are ready to prove the following theorem.
\begin{theorem}
\label{thm-c2}
Assume that $f$ satisfies \eqref{3I-20}-\eqref{3I-40}, and
\begin{equation}
\label{gj-I105}
 \lim_{R \rightarrow \infty} f (R {\bf 1}) =  \infty.
\end{equation}
Let $u \in \mathscr{U}$. If \eqref{3I-11s}, \eqref{A2}-\eqref{A4} and \eqref{f7} hold. Then
\begin{equation}
\label{eqn-c2g}
|\nabla^2 u|_{C^0 (\bM)} \leq C
\end{equation}
where $C$ depends on $|u|_{C^1(\bar M)}$, $|\ul u|_{C^2(\bar M)}$ and other known data.
\end{theorem}

\begin{proof}
Set
\[
W(x) = \max_{ \xi \in T_x M, |\xi| = 1} ( A^{\xi\xi} (x, u, \nabla u)+ \nabla_{\xi \xi} u) e^\phi, \;\; x \in \bM,
\]
where $\phi$ is a function to be determined. Assume that $W$ is achieved at an
interior point $x_{0} \in M$ in a unit direction $\xi \in T_{x_0} M$. Choose a smooth orthonormal local frame
$e_{1}, \ldots, e_{n}$ about $x_{0}$ such that $\xi = e_1$, $\nabla_i e_j (x_0) = 0$ and that
$U_{ij} (x_{0})$ is diagonal. We assume $U_{11} (x_0) > 0$ and
\[
U_{11} (x_0) \geq \cdots \geq U_{nn} (x_0).
\]

At the point $x_{0}$ where the function $\log U_{11} + \phi$ (defined near $x_{0}$) attains its maximum, we have
\begin{equation}
\label{gs3}
\frac{\nabla_{i} U_{11}}{U_{11}} + \nabla_i \phi = 0, \;\; i = 1, \cdots, n
\end{equation}
and
\begin{equation}
\label{gs4}
\frac{\nabla_{ii} U_{11}}{U_{11}}
   - \Big(\frac{\nabla_i U_{11}}{U_{11}}\Big)^2 + \nabla_{ii} \phi \leq 0.
\end{equation}
Differentiating equation \eqref{jw-21} twice and using \eqref{gs3}, we obtain at $x_{0}$,
\begin{equation}
\label{equa2}
\begin{aligned}
F^{ii}\nabla_{11} U_{ii} +& \, F^{ij,kl} \nabla_{1} U_{ij} \nabla_{1}U_{kl}\\
 \geq & \, \psi_{p_k} \nabla_k U_{11}
    + \psi_{p_k p_l} \nabla_{1k} u \nabla_{1l} u + \beta''_\varepsilon (u - h) |\nabla_1 (u - h)|^2 \\
    & + \beta'_\varepsilon (u - h) \nabla_{11} (u - h) - C U_{11}\\
  \geq & \, -U_{11} \psi_{p_k} \nabla_k \phi - C U_{11} + \psi_{p_1p_1} U_{11}^2 + (U_{11} - C) \beta'_\varepsilon (u - h)
  \end{aligned}
\end{equation}
provided $U_{11}$ is sufficiently large.
Recall the formula for interchanging order of covariant derivatives
\[
\begin{aligned}
\nabla_{ijkl} v - \nabla_{klij} v
= R^m_{ljk} \nabla_{im} v & + \nabla_i R^m_{ljk} \nabla_m v
      + R^m_{lik} \nabla_{jm} v \\
  & + R^m_{jik} \nabla_{lm} v
      + R^m_{jil} \nabla_{km} v + \nabla_k R^m_{jil} \nabla_m v.
\end{aligned}
\]
It follows
\begin{equation}
\label{chu2}
\begin{aligned}
F^{ii}\nabla_{ii}U_{11} \geq F^{ii}\nabla_{11} U_{ii} + F^{ii}(\nabla_{ii} A^{11} - \nabla_{11} A^{ii}) - C U_{11}\sum F^{ii}.
\end{aligned}
\end{equation}
Differentiating equation \eqref{jw-21} once, we obtain
\begin{equation}
\label{diffonce}
F^{ii} (\nabla_{kii}u+\nabla_k A^{ii}) = \nabla_k \psi + \nabla_k \beta_{\varepsilon}(u - h)
\end{equation}
Moreover, we use the formula
\begin{equation}
\label{hess-A70'}
 \nabla_{ikj} v - \nabla_{jik} v = R^l_{kij} \nabla_l v,
\end{equation}
to derive that
\begin{equation}
\label{cha}
\begin{aligned}
F^{ii}(\nabla_{ii}A^{11}-\nabla_{11} A^{ii})
\geq \,& F^{ii}(A^{11}_{p_k} \nabla _{iik} u - A^{ii}_{p_k} \nabla _{11k} u ) - CU_{11} \sum F^{ii} \\
           & +  F^{ii}(A^{11}_{p_i p_i} U^2_{ii} - A^{ii}_{p_1 p_1} U_{11}^2) - C \sum F^{ii} \\
\geq \,& U_{11} F^{ii} A_{p_k}^{ii} \nabla_k \phi - CU_{11} (1 + \sum F^{ii}) \\
           & -  C\sum_{i\geq 2} F^{ii} U_{ii}^2 - U_{11}^2 \sum_{i\geq 2} F^{ii} A_{p_1 p_1}^{ii} - C \beta'_{\varepsilon}(u - h).
\end{aligned}
\end{equation}
Thus, by substituting \eqref{chu2} into \eqref{gs4} and using \eqref{equa2} and \eqref{cha}, we obtain
\begin{equation}
\label{gs5}
\begin{aligned}
\mathcal{L} \phi
\leq \,& E - \psi_{p_1 p_1} U_{11} + \frac{C}{U_{11}} \sum F^{ii} U_{ii}^2 +  U_{11} \sum_{i\geq 2}F^{ii}A^{ii}_{p_1 p_1} \\
            & + \bigg(\frac{C}{U_{11}} - 1 \bigg) \beta'_{\varepsilon}(u - h) + C\sum F^{ii} + C
\end{aligned}
\end{equation}
where
\[
E = \frac{1}{U_{11}^2} F^{ii} (\nabla_i U_{11})^2
            + \frac{1}{U_{11}} F^{ij,kl} \nabla_{1} U_{ij} \nabla_{1} U_{kl}.
\]

Let
\[
\phi = \frac{\delta |\nabla u |^2}{2} + b (\ul u - u)
\]
where $b$, $\delta$ are undetermined constants satisfying $0 < \delta < 1 \leq b$. Direct computation yields
\[
\nabla_{i} \phi  =  \delta \nabla_k u \nabla_{ik} u + b \nabla_i (\ul u - u)
\]
and
\[
\begin{aligned}
\nabla_{ii} \phi
     = \,& \delta  (\nabla_{ik} u )^2 + \delta \nabla_k u  \nabla_{iik} u  + b \nabla_{ii} (\ul u - u)  \\
   \geq \,& \frac{\delta}{2} U_{ii}^2 - C \delta + \delta \nabla_k u \nabla_{iik} u + b \nabla_{ii} (\ul u - u).
\end{aligned}
\]
From \eqref{diffonce}, we have
\[
\begin{aligned}
 F^{ii} \nabla_k u \nabla_k U_{ii}
      = \,&  \nabla_k u \psi_{x_k} + \psi_u |\nabla u|^2 + \psi_{p_l} \nabla_k u \nabla_{kl}u \\
          & + \beta'_\varepsilon (u - h) (|\nabla u|^2 - \nabla u \cdot \nabla h).
\end{aligned}
\]
We then have by \eqref{hess-A70'} that
\[
F^{ii}\nabla_k u \nabla_{iik} u
          \geq (\psi_{p_l} -  F^{ii} A^{ii}_{p_l})\nabla_k u \nabla_{kl}u - C (1 + \sum F^{ii}) - C \beta'_{\varepsilon} (u - h).
\]
Therefore,
\begin{equation}
\label{gs8.5}
\mathcal{L}\phi \geq  b \mathcal{L}(\ul u - u) + \frac{\delta}{2} F^{ii} U_{ii}^2  - C \delta \beta'_{\varepsilon}(u - h)- C\sum F^{ii}-C.
\end{equation}

Now we estimate $E$ in \eqref{gs5} following \cite{GOLD} (see also \cite{U}) by using an
inequality shown by Andrews~\cite{A} and Gerhardt~\cite{GC}.
For fixed $0 < s \leq 1/3$, let
\[
J = \{i: U_{ii} \leq - s U_{11}\}, \;\;
K = \{i: U_{ii} > - s U_{11}\}.
\]
Similar to \cite{GOLD}, we have
\[
- F^{ij, kl} \nabla_1 U_{ij} \nabla_1 U_{kl}
         \geq \frac{2 (1-s)}{(1+s) U_{11}} \sum_{i \in K} (F^{ii} - F^{11})((\nabla_i U_{11})^2 - CU_{11}^2/s).
\]
Then,
\begin{equation}
\label{gj-S140}
\begin{aligned}
 E \leq \,& \frac{1}{U_{11}^2} \sum_{i \in J} F^{ii} (\nabla_i U_{11})^2
             + C \sum_{i \in K} F^{ii}
             + \frac{C F^{11}}{U_{11}^2}  \sum_{i \notin J} (\nabla_i U_{11})^2  \\
   \leq \,& \sum_{i \in J} F^{ii} (\nabla_i \phi)^2
             +  C \sum F^{ii} + C F^{11} \sum (\nabla_i \phi)^2 \\
   \leq \,& C b^2 \sum_{i \in J} F^{ii}  + C \delta^2 \sum F^{ii} U_{ii}^2
            +  C \sum F^{ii} + C (\delta^2 U_{11}^2 + b^2) F^{11}.
\end{aligned}
\end{equation}
Combining \eqref{gs5}, \eqref{gs8.5} and \eqref{gj-S140}, we obtain
\[
\begin{aligned}
b \mathcal{L}(\ul u - u)
\leq \,& \Big(C\delta^2-\frac{\delta}{2} + \frac{C}{U_{11}}\Big)F^{ii}U_{ii}^2 + C b^2\sum_{i \in J}F^{ii}+ C(1 + \sum F^{ii})\\
                & + C b^2 F^{11} + \Big( C \delta - 1 + \frac{C}{U_{11}}\Big) \beta'_{\varepsilon}(u-h).
\end{aligned}
\]
Taking $\delta<1$ small enough such that
\[
c_1:= -\frac{1}{2}\max \{C\delta^2-\frac{\delta}{2}, C \delta - 1\} > 0
\]
Then we may assume
\[
\max \{C\delta^2-\frac{\delta}{2} + \frac{C}{U_{11}}, \frac{C}{U_{11}} + C \delta -1\} \leq - c_1,
\]
otherwise, we have $U_{11} \leq C/c_1$ and we are done. Therefore,
\begin{equation}
\label{left}
\begin{aligned}
b \mathcal{L}(\ul u - u)
\leq  \,&  - c_1 F^{ii}U_{ii}^2 + C b^2\sum_{i \in J}F^{ii} \\
               & + C(1 + \sum F^{ii})+ C b^2 F^{11} - c_1 \beta'_{\varepsilon}(u-h).
\end{aligned}
\end{equation}

So far, the proof above follows essentially \cite{BDJ}.
From now on we use the new method introduced by Guan \cite{GNEW}.

Let $\tilde{\mu}=\mu(x_0)$ and $\tilde{\lambda}=\lambda(U(x_0))$.
If $|\nu_{\tilde{\mu}} - \nu_{\tilde{\lambda}}| \geq \zeta_0$, we apply \eqref{L-u} to \eqref{left} and obtain that
\begin{equation}
\label{L}
\begin{aligned}
( b \theta - C )( 1 + \sum F^{ii} )
\leq \,& - c_1 F^{ii}U_{ii}^2  + C b^2 F^{11} + C b^2\sum_{i \in J}F^{ii} \\
            & +  b \beta_{\varepsilon}( u - h ) - c_1 \beta'_{\varepsilon}(u - h).
\end{aligned}
\end{equation}
Fix $b>1$ sufficiently large such that $b \theta - C > 0$, and it follows from Lemma \ref{jw-lem1} that
\[
b \beta_{\varepsilon}( u - h ) - c_1 \beta'_{\varepsilon}(u - h)
      \leq \frac{(u - h)^2}{\varepsilon}( b (c_0 \varepsilon)^{1/3} - 3c_1) \leq 0
\]
if $\varepsilon \leq ( 3c_1/bc_0^{1/3})^3$. Then \eqref{L} yields
\[
c_1 F^{ii}U_{ii}^2 - C b^2\sum_{i \in J}F^{ii} - C b^2 F^{11} \leq 0
\]
when $\varepsilon$ is small. Note that $|U_{ii}| \geq s U_{11}$ for $i \in J$.
It follows that
\[
(c_1 s^2 U^2_{11} - C b^2) \sum_{i \in J} F^{ii} + (c_1 U^2_{11} - C b^2) F^{11} \leq 0
\]
This implies a bound $U_{11} (x_0) \leq C b^2/(c_1 s^2)$.

Next suppose $|\nu_{\tilde{\mu}} - \nu_{\tilde{\lambda}}| < \zeta_0$. We then obtain by applying \eqref{L'} to  \eqref{left}  that
\[
c_1 F^{ii} U_{ii}^2 \leq C b^2 (1 + \sum f^{ii}) + b \beta_{\varepsilon}( u - h ) - c_1 \beta'_{\varepsilon}(u - h).
\]
Again we can choose $\varepsilon$ small enough such that
$b \beta_{\varepsilon}( u - h ) - c_1 \beta'_{\varepsilon}(u - h) \leq 0$.
Thus we have by \eqref{F'},
\begin{equation}
\label{case2}
\frac{ c_1 \zeta_0 |\tilde{\lambda}|^2}{\sqrt{n}} \sum F^{ii} \leq c_1 F^{ii} U_{ii}^2 \leq C b^2 (1 + \sum F^{ii})
\end{equation}
where $|\tilde{\lambda}|^2 = \sum \tilde{\lambda}_i^2=\sum U_{ii}^2$.
By the concavity of $f$, we have
\[
\begin{aligned}
|\tilde{\lambda}| \sum f_i
\geq \,& \sum f_i \tilde{\lambda}_i + f(|\tilde{\lambda}| \mathbf{1}) - f(\tilde{\lambda})\\
\geq \,& f(|\tilde{\lambda}| \mathbf{1}) - \psi[u](x_0)- c_0 - \frac{1}{4 |\tilde{\lambda}|} \sum f_i \tilde{\lambda}_i^2 - |\tilde{\lambda}| \sum f_i
\end{aligned}
\]
where $c_0$ comes from Lemma \ref{jw-lem1}. Therefore,
\begin{equation}
\label{subb1}
\begin{aligned}
|\tilde{\lambda}|^2 \sum f_i
\geq \,& \frac{|\tilde{\lambda}|}{2}(f(|\tilde{\lambda}| \mathbf{1}) - \psi[u](x_0) - c_0) - \frac{1}{8} \sum f_i \lambda_i^2 \\
\geq \,& |\tilde{\lambda}| - \frac{1}{8} \sum f_i \lambda_i^2
\end{aligned}
\end{equation}
when $|\tilde{\lambda}|$ is large enough satisfying $f(|\tilde{\lambda}| \mathbf{1})\geq 2 + c_0 + \max_{x\in\bM}\psi[u]$ by \eqref{gj-I105}. Combining \eqref{case2} and \eqref{subb1} we have
\[
|\tilde{\lambda}|^2 \sum F^{ii} + |\tilde{\lambda}| \leq C b^2 (1 + \sum F^{ii}),
\]
which gives $|\tilde{\lambda}|\leq C b^2$.
\end{proof}

\section{Gradient estimates and existence}

For the gradient estimates, we need some growth conditions in usual and assume that
\begin{equation}
\label{A1}
\left\{ \begin{aligned}
    p \cdot \nabla_x A^{\xi \xi} (x, z, p) + |p|^2  A^{\xi \xi}_z (x, z, p)
         \,& \leq \bar{\omega}_1 (x, z) |\xi|^2 (1 + |p|^{\gamma_1}), \\
    p \cdot \nabla_x \psi (x, z, p)  + |p|^2 \psi_z (x, z, p)
         \,& \geq - \bar{\omega}_2 (x, z) (1 + |p|^{\gamma_2}),
  \end{aligned} \right.
\end{equation}
for some constants $0 < \gamma_1, \gamma_2  < 4$ and some continuous functions $\bar{\omega}_1, \bar{\omega}_2 \geq 0$.
In addition to \eqref{A1}, assume that
\begin{equation}
\label{3I-50}
f_{j} (\lambda) \geq \nu_{0} \left( 1 + \sum f_{i} (\lambda) \right) \;\;
 \mbox{ for any }
  \lambda \in \Gamma \mbox{ with } \lambda_j < 0,
\end{equation}
where $\nu_0$ is a uniform positive constant. Note that \eqref{3I-50} is commonly used in deriving gradient estimates, see e.g. \cite{G}, \cite{U} and references therein. We also
need the following growth conditions:
\begin{equation}
\label{A5}
 p \cdot D_p \psi (x, z, p),
    \; - p \cdot D_p A^{\xi \xi} (x, z, p)/|\xi|^2
     \leq \bar{\omega}(x, z)  (1 + |p|^{\gamma})
\end{equation}
and
\begin{equation}
\label{gj-G20**}
|A^{\xi \eta} (x, z, p)|
     \leq \bar{\omega} (x, z) |\xi||\eta| (1 + |p|^{\gamma}),
\;\; \forall \, \xi, \eta \in T_x \bM, \xi \perp \eta,
\end{equation}
for some constant $\gamma \in (0, 2)$ and some continuous function $\bar{\omega} \geq 0$.

\begin{theorem}
\label{obs-th1b}
Assume that \eqref{3I-20}-\eqref{3I-30}, \eqref{A2}, \eqref{f7} hold.
Let $u \in C^3 (\bM)$ be an admissible solution to \eqref{jw-21} with $u\geq \ul u$ on $\bM$.
Suppose that \eqref{A1}-\eqref{gj-G20**}.
Then for $\varepsilon$ sufficiently small, we have
\begin{equation}
\label{obs-gradient}
\max_{\bar{M}} |\nabla u| \leq C (1 + \max_{\partial M} |\nabla u|),
\end{equation}
where $C$ depends on $|u|_{C^0 (\bar{M})}$, $|\ul u|_{C^2 (\bar{M})}$ and other known data.
\end{theorem}

The gradient estimates \eqref{obs-gradient} can be derived as in \cite{BDJ} using condition \eqref{f7} in place of (2.6) in \cite{BDJ}. We outline the proof here for completeness, and the reader can refer to \cite{BDJ} for more details and another group of assumptions that guarantees \eqref{obs-gradient}.

Suppose $|\nabla u| \phi^{-1/2}$
achieves a maximum at an interior point $x_0 \in M$,
where $\phi$ a positive function to be determined. As in Section 3 we choose smooth orthonormal local frames $e_1, \ldots, e_n$
about $x_0$ such that $\nabla_{e_i} e_j = 0$ at $x_0$
and $\{U_{ij} (x_0)\}$ is diagonal. Set $w = |\nabla u|$. Then at $x_0$, we have
\begin{equation}
\label{g1}
\frac{\nabla_i w}{w} - \frac{\nabla_i \phi}{2\phi} = 0,
\end{equation}
\begin{equation}
\label{g2}
\frac{\nabla_{ii} w}{w} + \frac{|\nabla_{i} \phi|^2}{4 \phi^2}
   - \frac{\nabla_{ii} \phi}{2 \phi} \leq 0
\end{equation}
for $i = 1, \ldots, n$. We see that for each fixed $1 \leq i \leq n$, $w \nabla_i w = \nabla_{l} u \nabla_{il} u$,
and by \eqref{hess-A70'} and (\ref{g1}) that
\begin{equation}
\label{g3}
\begin{aligned}
w \nabla_{ii} w
    = \,& (\nabla_{lii}u+ R^{k}_{iil} \nabla_{k} u) \nabla_{l} u
          + \Big(\delta_{kl} - \frac{\nabla_{k} u \nabla_{l} u}{w^2} \Big)
            \nabla_{ik} u \nabla_{il} u \\
 \geq \,& \nabla_{l} u \nabla_l U_{ii} - \frac{w^2}{2\phi} ( A^{ii}_{p_{k}} \nabla_{k} \phi
          + 2 \phi A^{ii}_{u}) - \nabla_l u A^{ii}_{x_l} - C w^2,
\end{aligned}
\end{equation}
in which the inequality follows from that the last term in the first equality is non-negative.
Differentiating the equation \eqref{jw-21}, by (\ref{g1}), we have
\begin{equation}
\label{g11'}
\begin{aligned}
 F^{ii} \nabla_{l} u \nabla_l U_{ii}
      = \,&  \nabla_{l} u \psi_{x_l} + \psi_u |\nabla u|^2
          + \frac{ w^2}{2\phi} \psi_{p_k} \nabla_k \phi \\
          & + \beta'_\varepsilon (u - h) (|\nabla u|^2 - \nabla u \cdot \nabla h).
\end{aligned}
\end{equation}
Take $\phi = - u + \sup_M u + 1$.
By \eqref{A2},
\begin{equation}
\label{obs-G20}
 A^{ii} = A^{ii} (x, u, \nabla u)
\leq A^{ii} (x, u, 0) + A^{ii}_{p_k} (x, u, 0) \nabla_k u,
\end{equation}
which implies by \eqref{f7} that
\begin{equation}
\label{obs-g17}
- F^{ii} \nabla_{ii} \phi \geq  - K_0 (1+ \sum F^{ii}) - F^{ii} A^{ii} \geq - C (1 + |\nabla u|) \sum F^{ii}  - K_0.
\end{equation}
Thus, by plugging \eqref{g3}, \eqref{g11'} and \eqref{obs-g17} into \eqref{g2}, and applying \eqref{A1} and \eqref{A5}, we obtain
\begin{equation}
\label{obs-g18}
\begin{aligned}
0 \geq \,& C_0 F^{ii} |\nabla_{i} u|^2 - C (|\nabla u|^{\gamma_2 - 2} + |\nabla u|^{\gamma} + 1) \\
            & - C(1 + |\nabla u| + |\nabla u|^{\gamma} + |\nabla u|^{\gamma_1 - 2})\sum F^{ii},
\end{aligned}
\end{equation}
where $C_0 = \min_{\bM} 1/4\phi^2 > 0$ depends on $|u|_{C^0(\bM)}$.
We may assume $\nabla_{1} u (x_{0}) \geq |\nabla u (x_{0})|/n > 0$.
From \eqref{g1}, \eqref{obs-G20} and \eqref{gj-G20**}, we see that
\[
U_{11} \leq  - \frac{1}{2\phi} |\nabla u|^2  + C (1 + |\nabla u| + |\nabla u|^{\gamma}) < 0
\]
if $|\nabla u|$ is sufficiently large, which yields by \eqref{3I-50} that
\[
F^{11} \geq \nu_{0} \Big(1 + \sum F^{ii}).
\]
We then see from \eqref{obs-g18} that
\[
\begin{aligned}
0  \geq \,& \frac{C_0 \nu_{0}}{n^2} \Big(1 + \sum F^{ii} \Big) |\nabla u|^2
        - C (|\nabla u|^{\gamma_2 - 2} + |\nabla u|^{\gamma} + 1)\\
            & - C(1 + |\nabla u| + |\nabla u|^{\gamma} + |\nabla u|^{\gamma_1 - 2})  \sum F^{ii}.
\end{aligned}
\]
Thus $|\nabla u (x_0)| \leq C$ and the proof of \eqref{obs-gradient} is completed.

\smallskip

Finally, by applying Theorem 4.1 in \cite{BDJ} which gives uniform bounds for $|u|_{C^0(\bar M)}$ and $|\nabla u|_{C^0(\partial M)}$,
provided (i) $A (x, z, p) \equiv A (x, p)$ and $A^{\xi \xi} (x, p)$
is concave in $p$ for each $\xi \in T_x M$ or
(ii) $\mathrm{tr} A (x, z, 0) \leq 0$ when $z$ is sufficiently large and
\begin{equation}
\label{A6}
|A^{\xi \xi} (x, z, p)| \leq \bar{\omega} (x, z) |\xi|^2 (1+ |p|^2)
\end{equation}
for any $\xi \in T_x M$ when $|p|$ is sufficiently large, where $\bar{\omega} \geq 0$ is
a continuous function.
We thus have derived \eqref{maines}. Therefore the Evans-Krylov theorem \cite{Evans82}, \cite{Krylov83} and the Schauder theory \cite{GT} ensure the smooth regularity of admissible solutions of \eqref{jw-21}, while the existence is guaranteed by the continuity method \cite{GT} and the degree theory \cite{LiYY89}; we omit the proof here as it is standard and well known. We finally obtain a $C^{1, 1} (\bM)$ viscosity solution satisfying \eqref{11} and \eqref{111}, see \cite{BDJ,XB}, by approximation.

We conclude
\begin{theorem}
Suppose that \eqref{3I-20}-\eqref{3I-40}, \eqref{3I-11s}, \eqref{A2}-\eqref{A4}, \eqref{f7}, \eqref{gj-I105}, \eqref{A1}-\eqref{gj-G20**} hold.
Then there exists a viscosity solution $u\in C^{1, 1} (\bM)$ to the obstacle problem \eqref{11} and \eqref{111} under any of the following additional conditions: (i) $A (x, z, p) \equiv A (x, p)$ and $A^{\xi \xi} (x, p)$
is concave in $p$ for each $\xi \in T_x M$; (ii) \eqref{A6} and $\mathrm{tr} A (x, z, 0) \leq 0$ when $z$ is sufficiently large.
Furthermore, $u$ belongs to $C^{3,\alpha}$ on $\{x\in M: u(x) < h(x)\}$, for any $\alpha\in (0,1)$.
\end{theorem}


\bigskip



\small
\begin{thebibliography}{99}

\bibitem{A}
B. Andrews, {\it Contraction of convex hypersurfaces in Euclidean space}, Calc. Var. Partial Differential Equations {\bf 2} (1994), 151-171.

\bibitem{BDJ}
G.-J. Bao, W.-S. Dong, H.-M. Jiao, {\it Regularity for an obstacle problem of Hessian equations on Riemannian manifolds}, J. Differential Equations http://dx.doi.org/10.1016/j.jde.2014.10.001.

\bibitem{CM}
L.A. Caffarelli, R. McCann, {\it Free boundaries in optimal transport and Monge-Amp\`{e}re obstacle problems}, Ann. of Math. {\bf 171} (2010), 673-730.

\bibitem{CNS}
L.A. Caffarelli, L. Nirenberg, J. Spruck, {\it Dirichlet problem for nonlinear second order elliptic equations III, Functions of the eigenvalues of the Hessian}, Acta Math. {\bf 155} (1985), 261-301.

\bibitem{CIL}
M. Crandall, H. Ishii, P. Lions, {\it User's guide to viscosity solutions of second order partial differential equations}, Bull. Amer. Math. Soc. {\bf 27} (1992), 1-67.

\bibitem{Evans82}
L.C. Evans, {\it Classical solutions of fully nonlinear, convex, second order elliptic equations}, Comm. Pure Appl. Math. {\bf 35} (1982), 333-363.



\bibitem{GT}
D. Gilbarg, N.S. Trudinger, {\it Elliptic partial differential equations of second order}, Springer-Verlag, New York, 2nd edition, 1983.


\bibitem{GC0}
C. Gerhardt, {\it Hypersurfaces of prescribed mean curvature over obstacles}, Math. Z. {\bf 133} (1973), 169-185.

\bibitem{GC}
C. Gerhardt, {\it Closed Weingarten hypersurfaces in Riemannian manifolds}, J. Differential Geom. {\bf 43} (1996), 612-641.


\bibitem{G}
B. Guan, {\it The Dirichlet problem for Hessian equations on Riemannian manifolds}, Calc. Var. Partial Differential Equations {\bf 8} (1999), 45-69.


\bibitem{GOLD}
B. Guan, {\it Second order estimates and regularity for fully nonlinear elliptic equations on Riemannian manifolds}, Duke Math. J. {\bf 163} (2014), 1491-1524.



\bibitem{GNEW}
B. Guan, {\it The Dirichlet problem for fully nonlinear ellipitc equations on Riemannian manifolds}, arXiv:1403.2133.

\bibitem{GSSNEW}
B. Guan, S. Shi, Z. Sui, {\it On estimates for fully nonlinear parabolic equations on Riemannian manifolds}, arXiv:1409.3633.


\bibitem{GJNEW}
B. Guan, H.-M. Jiao, {\it The Dirichlet problem for Hessian type elliptic equations on Riemannian manifolds}, preprint.

\bibitem{Giusti}
E. Giusti, {\it Superfici minime cartesiane con ostaeoli diseontinui}, Arch. Ration. Mech. Anal. {\bf 35} (1969), 47-82.


\bibitem{J15}
H.-M. Jiao, {\it $C^{1,1}$ regularity for an obstacle problem of Hessian equations on Riemannian manifolds}, preprint.

\bibitem{JW}
H.-M. Jiao, Y. Wang, {\it The obstacle problem for Hessian equations on Riemannian manifolds}, Nonlinear Anal. {\bf 95} (2014), 543-552.


\bibitem{Krylov83}
N.V. Krylov, {\it Boundedly nonhomogeneous elliptic and parabolic equations in a domain}, Izvestia Math. Ser. {\bf 47} (1983), 75-108.


\bibitem{Kinderlehrer71}
D.S. Kinderlehrer, {\it Variational inequalities with lower dimensional obstacles}, Israel J. Math. {\bf 10} (1971), 339-348.

\bibitem{Kinderlehrer73}
D.S. Kinderlehrer, {\it How a minimal surface leaves an obstacle}, Acta Math. {\bf 130} (1973), 221-242.


\bibitem{L}
K. Lee, {\it The obstacle problem for Monge-Amp\`{e}re equation}, Comm. Partial Differential Equations {\bf 26} (2001), 33-42.



\bibitem{LiYY89}
Y.-Y. Li, {\it Degree theory for second order nonlinear elliptic operators and its applications}, Comm. Partial Differential Equations {\bf 14} (1989), 1541-1578.



\bibitem{LZ}
J.-K. Liu, B. Zhou, {\it An obstacle problem for a class of Monge-Amp\`{e}re type functionals}, J. Differential Equations {\bf 254} (2013), 1306-1325.



\bibitem{O}
A. Oberman, {\it The convex envelope is the solution of a nonlinear obstacle problem}, Proc. Amer. Math. Soc. {\bf 135} (2007), 1689-1694.

\bibitem{OS}
A. Oberman, L. Silvestre, {\it The Dirichlet problem for the convex envelope}, Trans. Amer. Math. Soc. {\bf 363} (2011), 5871-5886.


\bibitem{Savin04}
O. Savin, {\it A free boundary problem with optimal transportation}, Comm. Pure Appl. Math. {\bf 57} (2004), 126-140.

\bibitem{Savin05}
O. Savin, {\it The obstacle problem for Monge Ampere equation}, Calc. Var. Partial Differential Equations {\bf 22} (2005), 303-320.

\bibitem{T}
N.S. Trudinger, {\it The Dirichlet problem for the prescribed curvature equations}, Arch. Ration. Mech. Anal. {\bf 111} (1990), 153-179.

\bibitem{U}
J. Urbas, {\it Hessian equations on compact Riemannian manifolds, Nonlinear Problems in Mathematical Physics and Related Topics, II}, Kluwer/Plenum, New York, 2002, pp. 367-377.

\bibitem{XB}
J.-G. Xiong, J.-G. Bao, {\it The obstacle problem for Monge-Amp\`{e}re type equations in non-convex domains}, Commun. Pure Appl. Anal. {\bf 10} (2011), 59-68.

\end{thebibliography}
\end{document}